\documentclass[11pt]{scrartcl}
\usepackage[utf8]{inputenc}
\usepackage{amsmath,amssymb,amsthm}
\usepackage{enumerate}
\usepackage{url}
\newtheorem{thm}{Theorem}[section]
\newtheorem{dfn}[thm]{Definition}
\newtheorem{lem}[thm]{Lemma}
\newtheorem{prop}[thm]{Proposition}

\theoremstyle{remark}
\newtheorem{rem}[thm]{Remark}

\DeclareMathOperator{\scal}{Scal}
\DeclareMathOperator{\ric}{Ric}

\DeclareMathOperator{\tr}{trace}
\DeclareMathOperator{\Id}{I}
\DeclareMathOperator{\id}{id}

\newcommand{\R}{\mathrm{R}}

\newcommand{\ps}[2]{\left\langle #1,#2 \right\rangle}
\newcommand{\eps}{\varepsilon}
\newcommand{\rt}{\longrightarrow}
\newcommand{\C}{{\mathbb C}}
\newcommand{\re}{{\mathbb R}}

\date{}
\begin{document}

\title{Positive isotropic curvature and \\ self-duality in dimension 4}

\author{Thomas Richard and Harish Seshadri}

\maketitle

\begin{abstract}
  We study a positivity condition for the curvature of
  oriented Riemannian 4-manifolds: The half-$PIC$ condition. It is a slight
  weakening of the positive isotropic curvature ($PIC$) condition introduced by M. Micallef and
  J. Moore.

 We observe that the half-$PIC$ condition is preserved by the Ricci
  flow and satisfies a maximality property among all Ricci flow
  invariant positivity conditions on the curvature of oriented
  4-manifolds.

  We also study some  geometric and
  topological aspects of half-$PIC$ manifolds.
\end{abstract}

\section{Introduction}
Let $(M,g)$ be an oriented Riemannian 4-manifold. We say that $M$ has
{\it positive isotropic curvature on self-dual 2-forms} (which we
abbreviate as $PIC_+$) if the complex linear extension of the
curvature operator $\R: \Lambda^2 T_pM \otimes \C \to \Lambda^2 T_pM
\otimes \C$ satisfies
\[\ps{\R (v \wedge w)}{v \wedge w}>0\]
for every $p \in M$ and $v \wedge w \in \Lambda^2_+T_pM \otimes \C
\cap S_0$, where $\ps{\ }{\ }$ denotes the Hermitian inner product on
$\Lambda^2 T_pM \otimes \C$, $\Lambda^2_+ T_pM$ denotes the self-dual
2-forms at $p$ and  $S_0 \subset \Lambda^2T_p M\otimes\mathbb{C}$
denotes the set of complex $2$-forms that can be written as $v\wedge
w$ where $v,w\in T_pM\otimes\C$ span a complex plane which is
isotropic for the complex bilinear extension of the Riemannian metric
on $T_pM\otimes\C$.

This condition is a variation of the $PIC$ condition introduced by
Micallef and Moore in \cite{mm}, which requires positivity for all
$v\wedge w$ such $\{v,w\}$ spans an isotropic complex plane, without
requiring $v\wedge w$ to be self dual.

One has the following alternative definitions (see Proposition \ref{dc}): %\vspace{2mm}
\begin{enumerate}
\item  If ``$\scal$'' denotes scalar curvature, $\Id: \Lambda^2 TM \to
  \Lambda^2 TM$ is the identity operator and $ \R_{\mathcal{W}_+}: \Lambda^2_+TM \to \Lambda^2_+TM$ denotes the self-dual part of the Weyl tensor of $\R$ then
\[\tfrac{\scal}{6}\Id-\R_{\mathcal{W}_+}> 0\]
on $\Lambda^2_+TM.$

\item $M$ is $PIC_+$ if for any $p\in M$ and any \emph{oriented} orthonormal basis $(e_1,e_2,e_3,e_4)$ of $T_pM$, we have
%\[\ps{\R\left((e_1+ie_2)\wedge (e_3+ie_4)\right)}{(e_1+ie_2)\wedge (e_3+ie_4)}>0.\]
\[R_{1313}+R_{1414}+R_{2323}+R_{2424}-2R_{1234}>0\]
where $R_{ijkl}=\ps{\R(e_i\wedge e_j)}{e_k\wedge e_l}$.

\end{enumerate}

One has similar definitions in the nonnegative case, which we denote by $NNIC_{+}$ and the case of $\wedge^2_{-} T_pM$, which we denote by  $PIC_{-}$ and $NNIC_{-}$. If a metric is either $PIC_{+}$ or $PIC_{-}$ we say that it is {\it half-PIC}.
As examples we note that:
\begin{enumerate}
\item Any anti-self-dual $4$-manifold i.e., a manifold with $\R_{\mathcal{W}_+}=0$,  with positive scalar curvature is $PIC_+$.
\item All $4$-manifolds with positive isotropic curvature are $PIC_+$. This
  includes $\mathbb{S}^4$, $\mathbb{S}^3\times\mathbb{S}^1$. It is also known that connected sums of manifolds with positive isotropic curvature admit metric with positive isotropic curvature.
\item Any K\"ahler 4-manifold whose scalar curvature is nonnegative is
  $NNIC_+$ but not $PIC_+$ (with its standard orientation, see
  Proposition
  \ref{prop:kaehlerPIC+}). Thus $\mathbb{CP}^2$ is $NNIC_+$. Moreover,
  since $\mathbb{CP}^2$
  is half conformally flat (it satisfies $\R_{\mathcal{W}_-}=0$), it
  is also $PIC_-$.
\end{enumerate}
These examples show that the $PIC_+$ condition is strictly weaker than
the usual $PIC$ condition.

By applying Wilking's criterion ~\cite{wil} it is easy to see (see Proposition \ref{pre}) that the $NNIC_{\pm}$ conditions are preserved by Ricci flow. We were informed by S. Brendle that R. Hamilton knew this fact and the proof is similar to Hamilton's proof that $PIC$ is preserved by the Ricci flow in dimension 4 (The fact that  $PIC$ is prserved by Ricci flow in all dimensions is a relatively recent result due to  S. Brendle-R. Schoen \cite{bs1} and H.T. Nguyen  \cite{ngu}). The main result of this note is that the $NNIC_{\pm}$ conditions are the minimal Ricci flow invariant nonnegativity conditions in dimension $4$. More precisely, if $S_B^2 \Lambda ^2 \re^4$ denotes the space of algebraic curvature operators in dimension $4$, we have the following:

\begin{thm}
Let $\mathcal{C} \subsetneq \mathcal{C}_{\scal} \subset S_B^2 \Lambda ^2 \re^4$ be an oriented Ricci flow invariant curvature cone. $\mathcal{C}$ is then contained in one of the cones $NNIC_+$ or
$NNIC_-$. Moreover, if $\mathcal{C}$ is an unoriented curvature cone, it is contained in the $NNIC$ cone.
\end{thm}

This is proved in Section \ref{sec:half-pic-as}. We note that the above result applies to {\it all} Ricci flow invariant cones. In ~\cite{gms} a similar result is proved in dimensions $\ge 5$: {\it All the Ricci flow invariant cones constructed by Wilking ~\cite{wil} are contained in the $NNIC$ cone}.

In Section \ref{sec:rigid-einst-half}, we show that compact Einstein $4$-manifolds which are $PIC_+$ are rigid:
\begin{prop}\label{ein}
A compact oriented Einstein $PIC_-$ 4-manifold is isometric, up to scaling, to $\mathbb{S}^4$ or
$\mathbb{CP}^2$ with their standard metrics. Moreover, an Einstein $NNIC_-$ manifold is either
$PIC_-$ or flat or a negatively oriented K\"ahler-Einstein surface with nonnegative scalar curvature.
\end{prop}
The method of proof we use is based on a Bochner type formula for
Einstein curvature operators arising from Ricci flow, which was introduced by S. Brendle
in \cite{bre}. Our method of proof is based on a subsequent work of Brendle \cite{bre2}).  Another proof can be given using the equality case of a
theorem of M. Gursky and C. LeBrun in \cite{gl}.

In Section \ref{sec:topology-half-pic} we consider the topological
implications of the $PIC_+$ condition. These results are
obtained using standard geometric techniques, and their proofs are thus
only sketched.

The first observation is a modification of a result of Micallef-Wang \cite{mw}
\begin{prop}
Let $M_i$, $i=1,2$, be  compact oriented 4-manifolds admitting metrics with $PIC_{+}$. Then the connected sum $M_1 \# M_2$ admits
a metric with $PIC_{+}$.
\end{prop}
Next, a standard Bochner formula argument shows:
\begin{prop}
If an oriented $4$-manifold $(M^4,g)$ is compact and $PIC_+$, then
$b_2^+(M)=0$. If $(M,g)$ is $NNIC_+$ then $b_+(M)\leq 3$, moreover, if $b_+(M)\geq
  2$ then $(M,g)$ is either flat or isometric to a $K3$ surface.
\end{prop}

We conclude this introduction with some speculative remarks. It would be natural to try to extend Proposition
\ref{ein} to Ricci solitons. Note that such a result is known for $PIC$ solitons. More generally, it might be interesting to study if a Ricci flow with surgery procedure holds for $PIC_{+}$ manifolds. Note that
$4$-manifolds with positive isotropic curvature have been classified using Ricci flow with surgery
(see \cite{ham} and \cite{cz}).  A similar classification might  yield a diffeomorphism classification of $4$-manifolds with $PIC_{+}$.

\section{Elementary observations about the half-PIC cone}
\subsection{Curvature operators and curvature cones}
\begin{dfn}
  The space of algebraic curvature operators
  $S^2_B\Lambda^2\mathbb{R}^n$ is the space of symmetric endomorphisms
  $\R$ of $\Lambda^2\mathbb{R}^n$ which satisfy the first Bianchi identity:
$$ \ps{\R(x\wedge y)}{z\wedge t}
  +\ps{\R(z\wedge x)}{y\wedge t}+\ps{\R(y\wedge z)}{x\wedge t}=0$$
for all $x,y,z,t\in\mathbb{R}^n$.

\end{dfn}
\begin{rem}
  Here, as in the rest of the paper, $\mathbb{R}^n$ is endowed with
  its standard Euclidean structure, and the scalar product on $\Lambda^2\mathbb{R}^n$ is
  given by:
  \[\ps{x\wedge y}{z\wedge t}=\ps{x}{z}\ps{y}{t}-\ps{x}{t}\ps{y}{z}.\]
  Similarly if $(M,g)$ is a Riemannian manifold, $\Lambda^2TM$
  will be equipped with the inner product coming from the Riemannian metric on $TM$.
\end{rem}

 The Ricci morphism:
$\rho:S^2_B\Lambda^2\mathbb{R}^n\to S^2\mathbb{R}^n$ is defined by
\[\ps{\rho(\R)x}{y}=\sum_{i=1}^n \ps{\R(x\wedge e_i)}{y\wedge e_i}\]
where $(e_i)_{1\leq i\leq n}$ is an orthonormal basis of
$\mathbb{R}^n$. $\R$ is said to be Einstein if $\rho(\R)$ is a
multiple of the identity operator $\id:\mathbb{R}^n\to\mathbb{R}^n$. Similarly, the scalar curvature of an algebraic
curvature operator is just twice its trace.

The action of $O(n,\mathbb{R})$ on $\mathbb{R}^n$ induces the following action of $O(n,\mathbb{R})$ on $S^2_B\Lambda^2\mathbb{R}^n$:
\begin{equation}
\ps{g.\R(x\wedge y)}{z\wedge t}=\ps{\R(gx\wedge gy)}{gz\wedge gt}.\label{eq:action}
\end{equation}

Recall that the representation of $O(n,\mathbb{R})$ given by its
action on $S^2_B\Lambda^2\mathbb{R}^n$ is decomposed into irreducible
representations in the following way:
\begin{equation}
S^2_B\Lambda^2\mathbb{R}^n=\mathbb{R}\Id\oplus
(S^2_0\mathbb{R}^n\wedge\id)\oplus \mathcal{W}\label{eq:decomp}
\end{equation}
where the space of Weyl curvature operators $\mathcal{W}$ is the
kernel of the Ricci endomorphism $\rho$ and $S^2_0\mathbb{R}^n\wedge\id$ is the image of the space of traceless
endomorphims of $\mathbb{R}^n$ under the application $A_0\mapsto
A_0\wedge\id$. The wedge product of two symmetric operators
$A,B:\mathbb{R}^n\to \mathbb{R}^n$ is defined by
\[(A\wedge B)(x\wedge y)=\frac{1}{2}\left ( Ax\wedge By+Bx\wedge
  Ay\right).\]
This corresponds to the half of the Kulkarni-Nomizu product of $A$ and
$B$ viewed as quadratic forms.

In dimension 2, only the first summand of
\eqref{eq:decomp} exists. In dimension $3$ the $\mathcal{W}$
factor is $0$.

Starting in dimension 4, all three components exist. Moreover in
dimension 4, if we restrict the $O(4,\mathbb{R})$ action to an
$SO(4,\mathbb{R})$ action, the decomposition \eqref{eq:decomp} is no
longer irreducible: the Weyl part splits into self-dual and
anti-self-dual parts:
\begin{equation}
S^2_B\Lambda^2\mathbb{R}^4=\mathbb{R}\Id\oplus
(S^2_0\mathbb{R}^4\wedge\id)\oplus \mathcal{W}_+\oplus \mathcal{W}_-.\label{eq:decomp_dim4}
\end{equation}
This comes from the fact that the adjoint representation of $SO(4,\mathbb{R})$
on $\mathfrak{so}(4,\mathbb{R})\simeq\Lambda^2\mathbb{R}^4$ split into
two three dimensional irreducible components, which correspond to
self-dual and anti-self-dual $2$-forms. See the Appendix for more details.

We will write $\R=\R_{\Id}+\R_0+\R_\mathcal{W}$ to denote the
decomposition of a curvature operator along the three
components of \eqref{eq:decomp}. Similarly we will write
$\R=\R_{\Id}+\R_0+\R_{\mathcal{W}_+}+\R_{\mathcal{W}_-}$ for the
decomposition of $\R\in S^2_B\Lambda^2\mathbb{R}^4$ along \eqref{eq:decomp_dim4}.

\begin{dfn}
  An oriented (resp. unoriented) curvature cone is a closed convex cone $\mathcal{C}\subset S^2_B\Lambda^2\mathbb{R}^n$ such that:
  \begin{itemize}
  \item $\mathcal{C}$ is invariant under the action of
    $SO(n,\mathbb{R})$ (resp. $O(n,\mathbb{R})$) given by \eqref{eq:action}.
  \item The identity operator
    $\Id:\Lambda^2\mathbb{R}^n\to\Lambda^2\mathbb{R}^n$ is in the
    interior of $\mathcal{C}$.
  \end{itemize}
\end{dfn}

\begin{rem}
  The condition that $\Id$ is in the interior of $\mathcal{C}$ implies
  that $\mathcal{C}$ has full dimension.
\end{rem}

Each of these cones defines a nonnegativity condition for the curvature
of Riemannian manifold in the following way: the curvature operator
$\R$ of a Riemannian manifold $(M,g)$ is a section of the bundle $S^2_B\Lambda^2TM$
which is built from $TM$  the same way $S^2_B\Lambda^2\mathbb{R}^n$ is
built from $\mathbb{R}^n$. For each $x\in M$, one can choose a
orthonormal basis of $T_xM$ to build an isomorphism between
$S^2_B\Lambda^2T_xM$ and $S^2_B\Lambda^2\mathbb{R}^n$. Thanks to the
$O(n,\mathbb{R})$-invariance of $\mathcal{C}$, this allows us to embed
$\mathcal{C}$ in $S^2_B\Lambda^2T_xM$ in a way which is independent
of the basis of $T_xM$ we started with.

We then say that $(M,g)$ has
$\mathcal{C}$-nonnegative curvature if for any $x\in M$ the curvature
operator of $(M,g)$ at $x$ belongs to the previously discussed
embedding of $\mathcal{C}$ in $S^2_B\Lambda^2T_xM$. Similarly, $(M,g)$
is said to have positive $\mathcal{C}$-curvature if its curvature
operator at each point is in the interior of $\mathcal{C}$. By
definition, the sphere $\mathbb{S}^n$ has positive
$\mathcal{C}$-curvature for all curvature cones $\mathcal{C}$.

\subsection{Ricci flow invariant curvature cones}

Let $Q$ be the quadratic vector field on
$S^2_B\Lambda^2\mathbb{R}^n$ defined by
\[Q(\R)=\R^2+\R^\#.\]
Here, $\R^2$ is just the square of $\R$ seen as an endomorphism of
$\Lambda^2\mathbb{R}^n$.  $\R^\#$ is defined in the following way:
\[\ps{\R^\#\eta}{\eta}=-\frac{1}{2}\sum_{i=1}^{n(n-1)/2}\ps{\left[\eta,\R
    \Bigl (\left[\eta,\R \left
          (\omega_i\right)\right]\Bigr)\right]}{\omega_i}\]
where $(\omega_i)_{i=1\dots n(n-1)/2}$ is an orthonormal basis of
$\Lambda^2\mathbb{R}^n$ and the Lie bracket $[\ ,\ ]$ on $\Lambda^2\mathbb{R}^n$ comes from its
identification with $\mathfrak{so}(n,\mathbb{R})$ given by:
\[ \phi: x\wedge y\mapsto (u\mapsto \ps{x}{u}y-\ps{y}{u}x).\]
This expression for $\R^\#$ can be found in \cite{bw}.

\begin{dfn}
  A curvature cone $\mathcal{C}$ is said to be Ricci flow invariant if for any
  $\R \in \partial\mathcal{C}$ of $\mathcal{C}$,
  $Q(\R)\in T_{\R}\mathcal{C}$, the tangent cone at $\R$ to $\mathcal{C}$.
\end{dfn}
\begin{rem}
  This condition is equivalent to the fact that the solutions to the ODE
  $\frac{d}{dt}\R=Q(\R)$ which start inside $\mathcal{C}$ stay in
  $\mathcal{C}$ for positive times.
\end{rem}

We now discuss the relevance of this notion to  the actual Ricci flow equation. This is the content of Hamilton's maximum principle:

\begin{thm} (R. Hamilton \cite{ham})
  Let $n \ge 2$ and $\mathcal{C} \subset S^2_B\Lambda^2\mathbb{R}^n$ be a Ricci flow invariant curvature cone.
  If $(M^n,g(t))_{t\in[0,T)}$ is a Ricci flow on a compact manifold such
  that $(M,g(0))$ has $\mathcal{C}$-nonnegative curvature, then for
  $t\in[0,T)$, $(M,g(t))$ has $\mathcal{C}$-nonnegative curvature.
\end{thm}

Next we recall Wilking's method for generating Ricci flow invariant cones.  %\vspace{2mm}
We begin with the isomorphism $\phi: \wedge^2 \R^n
\rt {\bf so} (n, \R)$ given  by
\begin{equation}\notag
 \phi(u \wedge v)(x)= \langle u,x \rangle v - \langle v, x \rangle u \ \ \ x \in {\mathbb R}^n.
\end{equation}
Under the above identification of
$\Lambda^2\mathbb{R}^n$ with $\mathfrak{so}(n,\re)$ the inner product
on $\mathfrak{so} (n,\re)$ is given by  $\langle A, B \rangle = -
\frac{1}{2}tr (AB)$.  Extend this inner product to a Hermitian form on
$\mathfrak{so} (n,\re) \otimes_\re \C = \mathfrak{so} (n,\C)$. We also
extend any algebraic curvature operator $\R \in S^2(\mathfrak{so} (n,\re))$ to a complex linear map $ \mathfrak{so} (n,\C) \to  \mathfrak{so} (n,\C)$. Denoting the extensions by the same symbols, one has :%\vspace{2mm}
\begin{thm}[\cite{wil}]
  Let $S$ be a subset of the complex Lie algebra $\mathfrak{so} (n,\C)$. If
  $S$ is invariant under the adjoint action of the corresponding complex
  Lie group $SO(n,\C)$, then the curvature cone
\[\mathcal{C}(S)= \{ \R \in S^2({\mathfrak{so}(n,\mathbb{R})}) \ \vert \ \ps{\R(X)}{X}\ge 0 \ \text{ for  all}\ X \in S \}\]
is Ricci flow invariant.
\end{thm}

% Let $(M,g)$ be a Riemannian manifold and $p\in M$. Choosing an
% orthonormal basis of $T_pM$, we get an identification
% $\mathbb{R}^n\simeq T_pM$ which induces an identification
% $\Lambda^2\mathbb{C}^n\simeq \Lambda^2T_pM\otimes_\re\C$. This
% identification allows to embed
% $S$ inside $\Lambda^2T_pM\otimes_\re\C$. Moreover, thanks to the
% $O(n,\mathbb{R})$-invariance of $S$, the image of $S$ doesn't depend
% on the linear isometry between $T_pM$ and $\mathbb{R}^n$ we have used,
% we will denote th

We also have  Wilking's generalization of the Brendle-Schoen strong maximum principle \cite{bs2} (which is based on a maximum principle of J. M. Bony \cite{bon}): %\vspace{2mm}

\begin{thm}[\cite{wil}]
  Let $S$ be an $Ad_{SO(n,\C)}$ invariant subset of $\mathfrak{so}(n,\C)$
  and $(M,g)$ be a compact $n$-manifold with nonnegative
  $\mathcal{C}(S)$-curvature.  Let $g(t)$ be the solution to Ricci flow starting at
  $g$. For $p \in M$ and $t >0$, let $S_t(p) \subset \wedge^2T_p(M)\otimes_\re\C$
  be the subset corresponding to $S$ at time $t$ i.e., $S_t(p) =
  \rho_{g(t)}^{-1}(S)$ and let
\[T_t(p):=  \{ x \in S_t(p) \ : \ \langle \R(t)(x), x \rangle _t =0 \}.\]
Then the set $\bigcup_{p \in M} T_t(p)$ is invariant under parallel
transport.
\end{thm}

\subsection{The half-PIC cone}

\label{sec:elem-observ-about}

\begin{prop} \label{dc}
$\R \in S^2_B\Lambda^2\mathbb{R}^n$ is $PIC_+$ if and only if the symmetric operator on
$\Lambda^2_+\mathbb{R}^4$ defined by the quadratic form $\ps{\R\eta}{\eta}$ is
2-positive. This is also equivalent to the condition
$\tfrac{\scal}{6}-\R_{\mathcal{W}_+}> 0$.
\end{prop}
As a corollary,
we get that a $PIC_+$ curvature operator has positive scalar curvature and
a $NNIC_+$ curvature operator with zero scalar curvature has vanishing
$\R_{\mathcal{W}_+}$. Similar statements hold for $PIC_-$ and $NNIC_-$.
\begin{proof}
  This comes from Micallef and Moore \cite{mm}. Let $\omega=(e_1+ie_2)\wedge
  (e_3+ie_4)\in\Lambda^2\mathbb{C}^4$ where $(e_i)_{i=1\dots 4}$ is a
  oriented orthonormal basis of $\mathbb{R}^4$. Then :
  \begin{align*}
    \ps{\R\omega}{\omega}=&\ps{\R(e_1\wedge e_3-e_2\wedge e_4)}{e_1\wedge e_3-e_2\wedge e_4}\\
&+\ps{\R(e_1\wedge e_4+e_2\wedge e_3)}{e_1\wedge e_4+e_2\wedge e_3}.
  \end{align*}
  Since $e_1\wedge e_3-e_2\wedge e_4$ and $e_1\wedge e_4+e_2\wedge
  e_3$ are in $\Lambda^2_+\mathbb{R}^4$ and any pair of orthonormal
  2-forms in $\Lambda^2_+\mathbb{R}^4$ can be written in this way for
  some oriented orthonormal basis of $\re^4$, this shows that $\R$ is $PIC_+$ if and only
  if the quadratic form $\eta \mapsto\ps{\R\eta}{\eta}$ on $\Lambda^2_+\mathbb{R}^4$ is
  2-positive.

  To see the second equivalence, first notice that
  $\ps{\R\eta}{\eta}=\ps{(\R_{\Id}+\R_{\mathcal{W}_+})\eta}{\eta}$
  for any $\eta \in \Lambda^2_+\mathbb{R}^4$. So $\R$ is $PIC_+$ if and
  only if $\tfrac{\scal}{12}+\R_{\mathcal{W}_+}$ is 2-positive. Let
  $\lambda_+\leq\mu_+\leq\nu_+$ be the eigenvalues of
  $\R_{\mathcal{W}_+}$. $\tfrac{\scal}{12}+\R_{\mathcal{W}_+}$ is
  2-positive if and only if $(\tfrac{\scal}{12}+\mu_+)+(\tfrac{\scal}{12}+\lambda_+)>0$. The
  tracelessness of $\R_{\mathcal{W}_+}$ implies that $\R$ is $PIC_+$
  if and only if : $\tfrac{\scal}{6}-\nu_+>0$, which is exactly the
  condition $\tfrac{\scal}{6}-\R_{\mathcal{W}_+}>0$.
\end{proof}

\begin{prop}\label{pre}
  The cone $\mathcal{C}_{IC_+}=\{\R\in S^2_B\Lambda^2\mathbb{R}^4\ |\
  \text{$\R$ is $NNIC_+$}\}$ is a Ricci flow invariant curvature cone.
\end{prop}
\begin{proof}
  We remark that :
  \[\mathcal{C}_{IC_+}=\{\R\ |\ \forall \omega\in S,\
  \ps{\R\omega}{\omega}\geq 0\}\]
  with $S=\{\omega\in\mathfrak{so}(4,\mathbb{C})\ |\ \tr(\omega^2)=0,\ \omega\in\Lambda^2_+\mathbb{C}^4\}$. This comes from the fact that $\omega=\omega_1+i\omega_2\in S$ if and only if $|\omega_1|^2=|\omega_2|^2=1$ and $\ps{\omega_1}{\omega_2}=1$, and thus $\R$ is nonnegative on $S$ if and only if it is 2-nonnegative on $\Lambda^2_{+}\mathbb{R}^4$. Thus $\mathcal{C}$ is a Wilking
  cone and is therefore Ricci flow invariant (see \cite{wil}).
\end{proof}
\begin{rem}
  This result contradicts Theorem 1.1 in
  \cite{gms}, which says that any Wilking cone is
  contained in the cone of curvature operators with nonnegative
  isotropic curvature. The proof is however valid in dimension $n\geq
  5$. The point is that it uses (p. 4 of \cite{gms})
  the simplicity of the Lie algebra $\mathfrak{so}(n,\mathbb{C})$,
  which holds only if $n\neq 4$. However, on should notice that the
  cone $\mathcal{C}_{IC_+}$ is only $SO(4,\mathbb{R})$-invariant. We
  will prove in Section \ref{sec:half-pic-as}, that, regardless of the
  Wilking condition, the cone of curvature operators with nonnegative
  isotropic curvature is maximal among $O(4,\mathbb{R})$-invariant Ricci flow invariant cones
  strictly smaller than the cone of operators with nonnegative scalar curvature.
\end{rem}
The cone $\mathcal{C}_{IC_-}$ is defined similarly and is also Ricci
flow invariant. Moreover, we have that
$\mathcal{C}_{IC_+}\cap \mathcal{C}_{IC_-}$ is the cone of curvature
operators with nonnegative isotropic curvature.

An oriented $4$-manifold $(M,g)$ is said to be positive K\"ahler if $g$
is K\"ahler
for some almost complex structure $J$ and the orientation of $M$
induced by $J$ coincide with the given orientation of $M$. This is
equivalent to  the K\"ahler form being a section of the bundle
$\Lambda^2_+T^*M$. Similarly $(M,g)$ is
negative K\"ahler if $g$ is K\"ahler
for some almost complex structure $J$ and the orientation of $M$
induced by $J$ is the opposite of the given orientation of
$M$. With these condition, $\mathbb{CP}^2$ is positive K\"ahler.

\begin{prop}\label{prop:kaehlerPIC+}
  Any positive K\"ahler manifold with nonnegative scalar curvature is
  $NNIC_+$ but not $PIC_+$. A $NNIC_-$ positive K\"ahler manifold is biholomorphic to $\mathbb{CP}^2$.
\end{prop}
\begin{proof}
  A theorem of Derdzi\'nski (see \cite{der}) shows that if $(M^4,g)$ is positive K\"ahler and of
  real dimension 4, then
  the eigenvalues of $\R_{\mathcal{W}_+}$ are $\tfrac{\scal}{6}$,
  $-\tfrac{\scal}{12}$ and $-\tfrac{\scal}{12}$, so
  $\tfrac{\scal}{6}-\R_{\mathcal{W}_+}$ is nonnegative (with a zero
  eigenvalue). Hence $(M^4,g)$ is $NNIC_+$ and not $PIC_+$.

  If in addition $(M^4,g)$ is $NNIC_-$, then $(M^4,g)$ is K\"ahler and
  has nonnegative isotropic curvature, earlier results of one of the
  authors (\cite{ses})
  imply that $M^4$ is actually biholomorphic to $\mathbb{CP}^2$.
\end{proof}
\section{Half-PIC as a maximal Ricci flow invariant curvature condition}
\label{sec:half-pic-as}

In this section we prove:
\begin{thm}\label{sec:thm_cones}
Let $\mathcal{C} \subsetneq \mathcal{C}_{\scal} \subset S_B^2 \Lambda ^2 \re^4$ be an oriented Ricci flow invariant curvature cone. $\mathcal{C}$ is then contained in one of the cones $PIC_+$ or
$PIC_-$. Moreover, if $\mathcal{C}$ is an unoriented curvature cone, it is contained in the $PIC$ cone.
\end{thm}

We will need a couple of results from \cite{rs}.
\begin{prop}[Proposition A.5 from \cite{rs}]
If an oriented curvature cone  $\mathcal{C}$ contains a non zero
$\R\in\mathcal{W}_+$ (resp. $\R\in\mathcal{W}_-$), then $\mathcal{C}$
contains $\mathcal{W}_+$ (resp. $\mathcal{W}_-$).
\end{prop}
\begin{prop}[Proposition 3.6 from \cite{rs}]\label{prop:weyl}
  If an oriented curvature cone $\mathcal{C}$ contains $\mathcal{W}$
  and is Ricci flow invariant, then $\mathcal{C}$ is the cone
  $\mathcal{C}_{\scal}$ of curvature operators whose scalar curvature
  is nonnegative.
\end{prop}

% \begin{lem}
%   Let $\R\in S^2_B\Lambda^2\mathbb{R}^4$,
%   $\omega_+\in\Lambda^2_+\mathbb{R}^4$ and
%   $\omega_-\in\Lambda^2_-\mathbb{R}^4$. Then
%   $\ps{\R\omega_+}{\omega_+}=\ps{(\R_{\Id}+\R_{\mathcal{W}_+})\omega_+}{\omega_+}$
%   and
%   $\ps{\R\omega_-}{\omega_-}=\ps{(\R_{\Id}+\R_{\mathcal{W}_-})\omega_-}{\omega_-}$. In
%   particular, if $\R\in S^2_0\mathbb{R}^4\wedge\id$, $\ps{\R\omega_+}{\omega_+}=\ps{\R\omega_-}{\omega_-}=0$
% \end{lem}

We begin with a lemma:
\begin{lem}\label{sec:lem_proj}
Let $\mathcal{C}$ be an oriented curvature cone and $\R\in\mathcal{C}$, then $\R_{\Id}+\R_{\mathcal{W}_+}\in\mathcal{C}$.
\end{lem}
\begin{proof}
Consider the subgroup $G=\mathbb{S}^3_-\subset SO(4,\mathbb{R})$ defined
in the appendix. It
acts trivially on $\Lambda^2_+\mathbb{R}^4$ and irreducibly on $\Lambda^2_-\mathbb{R}^4$. Set
$\tilde{\R}=\int_G g.R dg$ where $dg$ is the Haar measure on
$G$. Since $\mathcal{C}$ is convex and invariant under
$SO(4,\mathbb{R})\supset G$,
$\tilde{\R}\in\mathcal{C}$. We claim that
$\tilde{\R}=\R_{\Id}+\R_{\mathcal{W}_+}$.

Since $G$ acts trivially on $\mathbb{R}\Id$ and
$\mathcal{W}_+$, we have $\tilde{\R}_{\Id}=\R_{\Id}$ and
$\tilde{\R}_{\mathcal{W}_+}=\R_{\mathcal{W}_+}$.

We now prove that $\tilde{\R}_0=0$. We have that $\R_0=A_0\wedge\id$
for some $A_0\in S^2_0\mathbb{R}^4$. Then
$g.\R_0=g^{-1}A_0g\wedge\id$. This implies that :
$\tilde{\R}_0=\tilde{A}_0\wedge\id$ where $\tilde{A}_0=\int_
G g^{-1}A_0gdg$. As described in the appendix, $G$, viewed as the group of
unit quaternions, acts on  $\mathbb{R}^4$, identified with the
quaternions $\mathbb{H}$, by left multiplication. This
implies that, for any $u\in\mathbb{R}^4$ :
\[\ps{\tilde{A}_0 u}{u}=\int_G \ps{A_0 (gu)}{gu}dg=\int_{\|x\|=\|u\|}
\ps{A_0 x}{x}dx\]
where $dx$ is the usual measure on the sphere of radius $\|u\|$ scaled
to have total mass $1$. Since the last integral is just a multiple of
the trace of $A_0$, we have shown that $\tilde{A}_0=0$.

It remains to show that
$\tilde{\R}_{\mathcal{W}_-}=0$. But the action of $G$
on the three dimensional space $\Lambda^2_-\mathbb{R}^4$ is isomorphic
to the standard action of $\mathbb{S}^3$ on $\mathbb{R}^3$ by
rotations. Thus the action of $G$ on $\mathcal{W}_-$ is isomorphic to
the irreducible action of $\mathbb{S}^3$ on traceless symmetric $3\times 3$
matrices. This implies that $\mathcal{W}_-$ is irreducible as a
reprensention of $G$, thus $\tilde{\R}_{\mathcal{W}_-}=\int_G g.\R_{\mathcal{W}_-}dg=0$.
\end{proof}
\begin{proof}[Proof of Theorem ~\ref{sec:thm_cones}.]
  We assume that $\mathcal{C}$ is contained in neither $\mathcal{C}_{IC_+}$ nor
  $\mathcal{C}_{IC_-}$ and we will show that $\mathcal{C}=\mathcal{C}_{\scal}$.
  By Proposition \ref{prop:weyl}. It is enough to show that $\mathcal{C}$ contains
  $\mathcal{W}$.

  Since $\mathcal{C}$ is not contained in $NNIC_+$, we
  can find some $\R\in\mathcal{C}$ such that $\R_{\mathcal{E}_+}=\R_{\Id}+\R_{\mathcal{W}_+}$ is
  not 2-nonnegative. Thanks to Lemma \ref{sec:lem_proj}, we can assume
  that $\R=\R_{\mathcal{E}_+}$. Let
  $v_1,v_2,v_3\in\Lambda^2_+\mathbb{R}^4$ be the eigenvectors of $\R$
  and $\mu_1\leq\mu_2\leq\mu_3$ be the associated eigenvalues. We have
  that $\mu_1+\mu_2+\mu_3>0$ and $\mu_1+\mu_2<0$.

  The action on $SO(4)$ on $\Lambda^2_+\mathbb{R}^4$ is transitive on
  oriented orthonormal basis, thus we can find some $g\in
  SO(4,\mathbb{R})$ such that $ge_1=e_2$, $ge_2=-e_1$ and
  $ge_3=e_3$. We consider the curvature operator
  $\tilde{R}=\tfrac{1}{2}(R+g.R)$. It has eigenvectors $e_1$, $e_2$
  and $e_3$ with associated eigenvalues
  $\tilde{\mu}_1=\tilde{\mu}_2<0<\tilde{\mu}_3$. From this we deduce
  that up to scaling $\tilde{\R}+|\tilde{\mu}_1|\Id$ is the curvature
  operator of $\overline{\mathbb{CP}}^2$. Thus the curvature operator
  of  $\overline{\mathbb{CP}}^2$ is in the interior of
  $\mathcal{C}$. Arguing as in the proof of corollary 0.3 in \cite{rs}, we get
  that $\mathcal{W}_+\subset\mathcal{C}$.

  Similarly, using that $\mathcal{C}$ contains some $\R$ which is not
  in $NNIC_-$, one can show that the curvature operator of
  $\mathbb{CP}^2$ is in the interior of $\mathcal{C}$ and get that $\mathcal{W}_-\subset\mathcal{C}$. We have
  thus proved that $\mathcal{W}\subset\mathcal{C}$. This proves the first
  part of the theorem.

  The unoriented case follows: any orientation
  reversing element of  $O(4,\mathbb{R})$ will exchange
  $\Lambda^2_+\mathbb{R}^4$ and $\Lambda^2_-\mathbb{R}^4$,
  consequently it will exchange the cones $PIC_+$ and $PIC_-$. Thus a
  cone which is invariant under the full $O(4,\mathbb{R})$ will have
  to be contained in the intersection of the $PIC_+$ and $PIC_-$
  cones, which is exactly the $PIC$ cone.
\end{proof}

\section{The topology of half-PIC manifolds}
\label{sec:topology-half-pic}

\begin{thm}
The connected sum of two $PIC_+$ manifolds admits a $PIC_+$ metric.
\end{thm}
\begin{proof}
  This follows from \cite{gms} since the curvature
  operator of $\mathbb{R}\times\mathbb{S}^3$ is in the interior of
  $\mathcal{C}_{IC_+}$. See also the recent work of Hoelzel \cite{hoe}.
\end{proof}

The following Bochner-Weitzenbock formula for self-dual $2$-forms $\omega_+$ can
be found in \cite [appendix C]{freed} :
  \[\Delta_d\omega_+=\Delta\omega_++2\left(\frac{\scal}{6}-\R_{\mathcal{W}_+}\right
  )\omega_+.\]

From this we easily deduce:
\begin{thm}
Let $(M,g)$ be a compact oriented 4-manifold.

\begin{itemize}
\item If $(M,g)$ is $PIC_+$ then $b_+(M)=0$.
\item   If $(M,g)$ is $NNIC_+$ then $b_+(M)\leq 3$, moreover, if $b_+(M)\geq
  2$ then $(M,g)$ is either flat or isometric to a $K3$ surface.
\end{itemize}
\end{thm}

\section{Rigidity of Einstein half-PIC manifolds}
\label{sec:rigid-einst-half}
We recall the classical rigidity theorem for half-conformally flat Einstein $4$-manifolds:

\begin{thm}(N. Hitchin ~\cite{hit}, T. Friedrich - H.Kurke ~\cite{fk})\label{hit}
Let $(M,g)$ be a compact oriented half-conformally flat Einstein $4$-manifold.
\begin{enumerate}
\item If $M$ has positive scalar curvature then it is isometric, up to
  scaling, to $S^4$ or $\C P^2$ with their canonical metrics.
\item If $M$ is scalar flat then it is either flat or its universal cover
  is isometric to a $K3$ metric with its Calabi-Yau metric.
\end{enumerate}
\end{thm}

The following result can be regarded as a generalization of the above theorem. This result also follows from the
work of C. LeBrun - M. Gursky in \cite{gl}.

\begin{thm}
An Einstein $PIC_-$ 4-manifold is isometric, up to scaling, to $\mathbb{S}^4$ or
$\mathbb{CP}^2$ with their standard metrics. Moreover, an Einstein $NNIC_-$ manifold is either
$PIC_-$ or flat or a negatively oriented K\"ahler-Einstein surface with nonnegative scalar curvature.
\end{thm}
\begin{proof}
  We first deal with the case where $(M^4,g)$ is $PIC_-$. By Theorem \ref{hit} it is enough to
  show that $\R_{\mathcal{W}_-}=0$. The proof of this fact closely follows the
  proof by Brendle in \cite{bre2} that Einstein
  manifolds which have certain Ricci flow invariant positive curvature properties are spherical space-forms.

  We can assume that $\ric=3g$. Let $\kappa>0$ be the largest real number such that
  $\tilde{\R}=\R-\kappa\Id$ is $NNIC_-$. Note that the scalar
  curvature of $\tilde{R}$ is a constant equal to $12(1-\kappa)$, in
  particular, $\kappa\in (0,1]$.
  Since $(M,g)$ is Einstein, Proposition 3 in \cite{bre} shows that:
  \[\Delta \R+2Q(\R)=6\R.\]
  We compute :
  \begin{align*}
    \Delta
    \tilde{\R}+2Q(\tilde{\R})&=\Delta\R+2Q(\R)-4\kappa B(\R,\Id)+2\kappa^2Q(\Id)\\
    &=6\R-12\kappa\Id+6\kappa^2\Id\\
    &=6\tilde{\R}+6\kappa(\kappa-1)\Id.
  \end{align*}
  Here we used that since $\R$ is Einstein and has scalar curvature
  $12$, $B(\R,\Id)=B(\R_{\Id},\Id)=B(\Id,\Id)=Q(\Id)=3\Id$ (these
  identities come from \cite{bw}).

  Note that at each point we have that $\Delta\tilde{\R}$, $Q(\tilde{\R})$ and
  $-\tilde{\R}$ are in $T_{\tilde{\R}}\mathcal{C}$, where $\mathcal{C}$
  is the cone of $NNIC_-$ curvature operators. This follows from the
  fact that $\mathcal{C}$ is a convex Ricci flow invariant cone. This
  implies that $6\kappa(\kappa-1)\Id\in T_{\tilde{\R}}\mathcal{C}$.

  We claim that the scalar curvature of $\tilde{\R}$ is zero, this
  will imply that $\tilde{\R}_{\mathcal{W}_-}=\R _{\mathcal{W}_-}$
  vanishes.

  Assume that the scalar curvature of $\tilde{\R}$ is positive. This
  implies that $\kappa<1$.  Consequently $-\Id\in
  T_{\tilde{\R}}\mathcal{C}$, and we can find $\eps>0$ such that
  $\tilde{\R}-\eps\Id\in\mathcal{C}$. Thus
  $\tilde{\R}=(\tilde{\R}-\eps\Id)+\eps\Id$ is in the interior of
  $\mathcal{C}$ since $\Id$ is in the interior. This contradicts the
  definition of $\kappa$.

  We now deal with the borderline cases. First, since the only locally
  reducible Einstein 4-manifold with nonnegative scalar curvature are
  quotients of $\mathbb{R}^4$ and $\mathbb{S}^2\times\mathbb{S}^2$,
  and both of them are negative K\"ahler, we can assume that $(M,g)$ is
  irreducible. So the only thing to show is that if $(M,g)$ is
  irreducible $NNIC_-$ and not negative K\"ahler, then it is $PIC_-$.

  Using Berger's classification of holonomy groups, we see that $(M,g)$ is irreducible
  not negative K\"ahler if and only if the holonomy group of $(M,g)$
  $G\subset SO(4,\mathbb{R})$ contains the group
  $\mathbb{S}^3_-\subset SO(4,\mathbb{R})$ defined in the appendix. This implies
  in particular that $G$ acts irreducibly on
  $\Lambda^2_-\mathbb{R}^4$.

  It follows from the strong maximum principle that the set :
  \[K=\{\omega\in\Lambda^2_-T^*M\ |\ \tfrac{\scal}{6}|\omega|^2-\ps{\R_{\mathcal{W}_-}\omega}{\omega}=
  0\}\] is invariant under parallel transport. The irreducibility of the
  action of $G$ implies that $K$ is either $0$, in which case $(M,g)$
  is $PIC_-$, or the whole $\Lambda^2_-T^*M$. This implies that
  $\scal=0$ and $\R_{\mathcal{W}_-}=0$ and hence, by Theorem \ref{hit},
  $(M,g)$ is a negative $K3$ surface, hence negative K\"ahler, a contradiction.
\end{proof}

% \section{Integrally half-PIC manifolds}
% \label{sec:integrally-half-pic}
% We show that a metric is conformal to a half-PIC metric if and only if
% it satisfies the half-PIC condition in an integral sense.
% \begin{prop}
%   Let $(M^4,g)$ be an oriented compact 4-manifold and $\nu_+(x)$ be the greatest eigenvalue of $\R_{\mathcal{W}_+}$ at
%   $x\in M$, then $g$ is conformal to a $PIC_+$ metric if and only if :
%   \[\int_M \frac{\scal}{6}-\nu_+dv_g>0.\]
% \end{prop}

\appendix

\section{$SO(4,\mathbb{R})$, $\Lambda^2\mathbb{R}^4$ and the curvature of $4$ manifolds}
\label{sec:so4-mathbbr-lambd}
In this section we recall a small number of classical facts rekated to
the splitting of $\Lambda^2TM$ in dimension 4. For practical purposes,
we interpret them using quaternions.

We identify $\mathbb{R}^4$ with the quaternions $\mathbb{H}$ by
$(x,y,z,t)\mapsto x+iy+jz+kt$. We denote by $\mathbb{S}^3$ the group
(isomorphic to $SU(2)$) of unit quaternions. The classical double cover
of $SO(4,\mathbb{R})$ by $\mathbb{S}^3\times \mathbb{S}^3$ is given by
:
\begin{align*}
  \pi:\mathbb{S}^3\times \mathbb{S}^3\to\ & SO(4,\mathbb{R})\\
  (q_1,q_2)\mapsto\ &(x\in\mathbb{H}\mapsto q_1xq_2^{-1}).
\end{align*}
This shows that the Lie algebra of $SO(4,\mathbb{R})$ (which is
isomorphic to $\Lambda^2\mathbb{R}^4$) splits as a direct sum :
\[\Lambda^2\mathbb{R}^4=\mathfrak{so}(3,\mathbb{R})\oplus \mathfrak{so}(3,\mathbb{R})\]
This decomposition is exactly the same as the decomposition
$\Lambda^2\mathbb{R}^4=\Lambda^2_-\mathbb{R}^4\oplus
\Lambda^2_+\mathbb{R}^4$ of two-forms into anti-self-dual and
self-dual parts.

We denote by $\mathbb{S}^3_+\subset SO(4,\mathbb{R})$ (resp. $\mathbb{S}^3_-\subset SO(4,\mathbb{R})$) the image of
$\{1\}\times\mathbb{S}^3$ (resp. $\mathbb{S}^3 \times \{1\}$) under
$\pi$. $\mathbb{S}^3_+$ acts irreducibly on $\Lambda^2_+\mathbb{R}^4$
and trivially on $\Lambda^2_-\mathbb{R}^4$.

From this decomposition, the following decomposition of the
representation of $SO(4,\mathbb{R})$ on the space
$S^2_B\Lambda^2\mathbb{R}^4$ of algebraic curvature operators holds :
\[S^2_B\Lambda^2\mathbb{R}^4=\mathbb{R}\Id\oplus
S^2_0\mathbb{R}^n\wedge\id\oplus \mathcal{W}_+\oplus\mathcal{W}_-,
\]
where $\Id:\Lambda^2\mathbb{R}^4\to\Lambda^2\mathbb{R}^4$ is the
identity operator, $S^2_0\mathbb{R}^n\wedge\id$ is the space of pure
traceless Ricci tensor, and $\mathcal{W}_+$ (resp. $\mathcal{W}_-$) is
the space of self-dual (resp. anti-self-dual) Weyl curvature
tensors. Note that $\mathcal{W}_+$ (resp. $\mathcal{W}_-$) is in fact the space of traceless
symmetric operators on $\Lambda^2_+\mathbb{R}^4$
(resp. $\Lambda^2_-\mathbb{R}^4$). If $\R$ is a curvature operator, we
will write its decomposition as $\R=\R_{\Id}+\R_0+\R_{\mathcal{W}_+}+\R_{\mathcal{W}_-}$.

% Let $J$ and $\tilde{J}$ be the two almost complex structures on
% $\mathbb{H}$ defined by $Jx=ix$ and $\tilde{J}x=xi$. Note that, viewed
% as 2-forms via the metric, $J$ is self-dual whereas $\tilde{J}$ is
% anti-self-dual. Also note that $\mathbb{S}^3_+$ is the special unitary
% group of $\mathbb{H}$ with the almost complex structure $J$ and
% $\mathbb{S}^3_-$ is the special unitary
% group of $\mathbb{H}$ with the almost complex structure $\tilde{J}$.

\end{document}